\documentclass[a4]{amsart}
\usepackage{amscd,amsthm,amsfonts,amssymb,amsmath}
\usepackage{enumerate}
\usepackage{geometry}
\usepackage{graphicx}  
\usepackage{caption}   
\usepackage{float}     
\usepackage{subcaption}

\newcommand{\Prob}{{\mathrm{Prob}}}
\newcommand{\Mat}{{\mathrm{Mat}}}
\newcommand{\Aut}{{\mathrm{Aut}}}
\newcommand{\alt}{{\mathrm{alt}}}
\newcommand{\her}{{\mathrm{Her}}}
\newcommand{\rk}{\mathrm{rk}}

\newcommand{\Cl}{\mathrm{Cl}}
\newcommand{\col}{\mathrm{col}}
\newcommand{\scs}{\mathrm{scs}}

\newcommand{\coker}{\mathrm{coker~}}
\newcommand{\rank}{\mathrm{rank~}}
\newcommand{\corank}{\mathrm{corank~}}
\newcommand{\sym}{\mathrm{sym}}

\newcommand{\sign}{{\mathrm{sign}}}

\newcommand{\BC}{{\mathbb {C}}}
\newcommand{\BF}{{\mathbb {F}}}
\newcommand{\BQ}{{\mathbb {Q}}}

\newcommand{\BZ}{{\mathbb {Z}}}

\newcommand{\ra}{\rightarrow}
\newcommand{\HS}{\mathcal{HS}}
\newcommand{\GL}{\mathrm{GL}}
\newcommand{\T}{\mathrm{T}}

\newtheorem{theorem}{Theorem}
\newtheorem{corollary}[theorem]{Corollary}
\newtheorem{lemma}[theorem]{Lemma}

\newtheorem{definition}[theorem]{Definition}
\theoremstyle{remark} \newtheorem{remark}[theorem]{Remark}
\newtheorem{ques}[theorem]{Question}

\title{Error term in the Cohen-Lenstra heuristic via random matrix approach}

\author{Yue Xu}
\address{School of Mathematics and Statistics, Xidian University, 266 Xinglong Section of Xifeng Road, Xi'an, Shaanxi 710126, China}
\email{xuyue@xidian.edu.cn}

\author{Xiuwu Zhu}
 \address{Beijing Institute of Mathematical Sciences and Applications, Beijing 101408, China.}
 \address{Yau Mathematical Sciences Center, Tsinghua University, Beijing 100084, China;}
\email{xwzhu@bimsa.cn}


\date{\today}

\begin{document}

\begin{abstract}
The Cohen-Lenstra heuristic predicts the distribution of ideal class groups over number fields. Random matrix models provide a natural framework for explaining this heuristic, and recent results demonstrate the effectiveness of these tools. In this paper, we extend the analysis of the random matrix model to examine the error term in the Cohen-Lenstra heuristic. Additionally, we derive the asymptotic distribution of the corank of random matrices over finite fields, which can be modeled as a special class of Markov chains.
\end{abstract}
\maketitle

\section{Introduction}
\subsection{Cohen-Lenstra heuristic}
The Cohen-Lenstra-Martinet heuristics \cite{martinet1990Étude, wang2021Moments} predict that for a family of number field extensions over a fixed base field, the distribution of ideal class groups is inversely proportional to the complexity of the algebraic structures of these groups, particularly the size of their automorphism groups. 
For example, $\BZ/9\BZ$ is expected to occur more frequently as a class group than $(\BZ/3\BZ)^2$.

In this paper, we focus on quadratic fields, following Cohen and Lenstra's original formulation \cite{cohen1984Heuristics}. 
Let $D$ be a fundamental discriminant and $\Cl(D)$ the ideal class group of $\BQ(\sqrt{D})$. For any odd prime $p$ and finite abelian $p$-group $G$, they conjectured:
\[
\lim_{X\to \infty}\frac{\#\{0<\pm D<X: \Cl(D)[p^\infty]\simeq G\}}{\#\{0<\pm D<X\}} = \frac{\eta_\infty(p)/\eta_{u_\pm}(p)}{|G|^{u_\pm}|\Aut(G)|},
\]
where $u_+=1$, $u_-=0$, and $\eta_i(p)=\prod_{j=1}^{i}(1-p^{-j})$ for $i=0,1,\dots,\infty$.
As a corollary,
\[
\sum_{0<\pm D<X}|\Cl(D)[p]|\sim C_\pm\sum_{1<\pm D<X}1\sim C'_{\pm}X \quad \text{as }X\to \infty
\]
for constants $C_\pm$ and $C'_\pm$. Davenport and Heilbronn \cite{davenport1971density} established the $p=3$ case in 1971 with $C_+=4/3$, $C_-=2$. For general $p$, recent work \cite[equation (1.14)]{koymans2024Bounds} shows that for any $\epsilon>0$,
\[
\sum_{0<\pm D<X}|\Cl(D)[p]|\ll_{p, \epsilon} X^{\frac{3}{2}-\frac{1}{p+1}+\epsilon},
\]
which remains far from the conjectured result.

For the $2$-part of $\Cl(D)$, Gauss's genus theory and the Hardy-Ramanujan theorem \cite{hardy1917Normal} imply that $\dim_{\BF_2} \Cl(D)[2]$ grows like $\log\log |D|$. 
Consequently, the $2$-torsion subgroup is of density zero as $|D|$ increases. 
Gerth \cite{gerth19844class} extended the conjecture to finite abelian $2$-groups $G$:
\[
\lim_{X\to \infty}\frac{\#\{0<\pm D<X: 2\Cl(D)[2^\infty]\simeq G\}}{\#\{0<\pm D<X\}} = \frac{\eta_\infty(2)/\eta_{u_\pm}(2)}{|G|^{u_\pm}|\Aut(G)|}.
\]
Smith \cite{smith2017$2^infty$selmer} proved this for imaginary quadratic fields in 2017, and recently extended these results to $\ell^\infty$-class groups of cyclic $\ell$-extensions over general base fields excluding $2\ell$-th roots of unity \cite{smith2023Distributiona, smith2023Distribution}.

\medskip

We now examine the error term in the Cohen-Lenstra heuristic.

\subsection{Error term and random matrix model}
The counting of fundamental discriminants is well-understood (for example, see \cite[equation (16)]{fouvry20074rank}):
\[
\#\{0<\pm D<X\} = \frac{3}{\pi^2}X + O(X^{1/2}).
\]
For any finite abelian $p$-group $G$, define the error term:
\[
E_{\pm, p}(G, X) := \#\{0<\pm D<X: \Cl(D)[p^\infty]\simeq G\} - \frac{\eta_\infty(p)/\eta_{u_\pm}(p)}{|G|^{u_\pm}|\Aut(G)|}\cdot \frac{3}{\pi^2} X.
\]
Smith's work \cite{smith2023Distributiona, smith2023Distribution} established the bound:
\[
E_{\pm, 2}(G, X) \ll X\exp\left(-c\cdot(\log\log\log X)^{1/2}\right),
\]
which naturally raises several questions about the error term's behavior:
\begin{ques}\label{Q:1}
Does $E_{\pm, p}(G, X)$ admit a power-saving bound (i.e., $O(X^\theta)$ for some $\theta<1$)? 
If so, does $\theta$ depend on $G$ or $p$;
moreover, can we determine an explicit main term for $E_{\pm, p}(G, X)$ as $X\to\infty$?
\end{ques}

For a function $f$ defined on all finite abelian $p$-groups, define the $f$-average error:
\[
E_{\pm, p}(f, X) := \sum_{G} f(G)\cdot E_{\pm, p}(G, X).
\]
\begin{ques}
Do the error terms $E_{\pm, p}(f,X)$ share the same properties as in Question~\ref{Q:1}?
\end{ques}

For the case $p=3$ and $f(G)=|G[3]|$, Bhargava, Taniguchi, and Thorne \cite{bhargava2024Improved} refined the Davenport-Heilbronn results, proving the existence of constants $B_\pm$ such that for any $\epsilon>0$:
\[
E_{\pm, 3}(|G[3]|, X) = B_\pm X^{5/6} + O(X^{2/3+\epsilon}).
\]

For general $p$, taking $f=\mathbf{1}_{\{G:\ G\text{ nontrivial}\}}$ (the indicator function for nontrivial groups), based on numerical experiments, Lewis and Williams \cite{lewis2019Numerical} conjectured that
\[
E_{+,p}(\mathbf{1}_{\{G:\ G \text{ nontrivial}\}}, X) \sim C_p X^{s_p},
\]
where $C_p$ depends on $p$, and $s_p$ (potentially consistent across odd primes) appears to lie between $0.7$ and $0.8$.

In subfigures (a)-(d) of Figure \ref{fig:error_power}, the prime $p$ is set to $3$, $5$, $7$, and $11$, and the elementary divisors of the $p$-group $G$ are $[1]$, $[p]$, $[p^2]$, and $[p, p]$, which correspond to $G \simeq 0$, $\BZ/p\BZ$, $\BZ/p^2\BZ$, and $(\BZ/p\BZ)^2$, respectively. 
We plot the ratios $\log |E_{-, p}(G, X)|/\log X$  as $X$ (the bound on the absolute discriminant of imaginary quadratic fields) increases.


\begin{figure}[H]
\centering
\subfloat[$p=3$]{\includegraphics[scale=0.48]{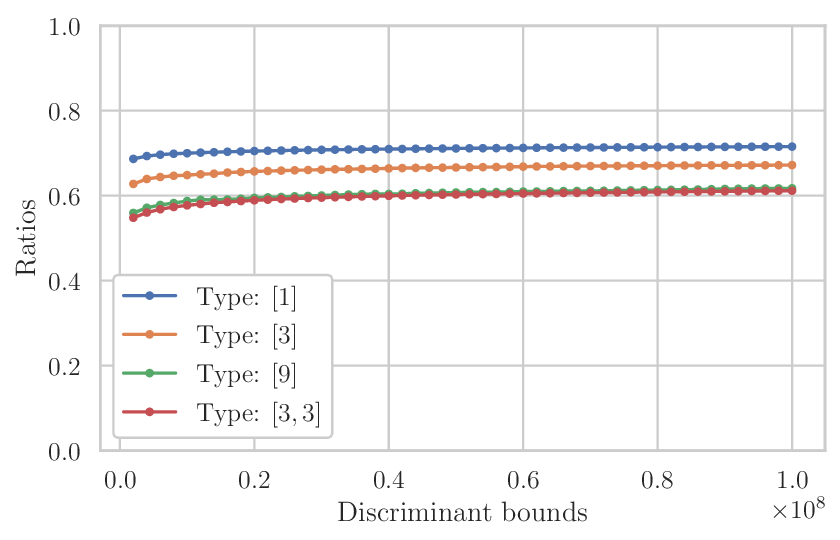}\label{fig:p3}}
\hfill
\subfloat[$p=5$]{\includegraphics[scale=0.48]{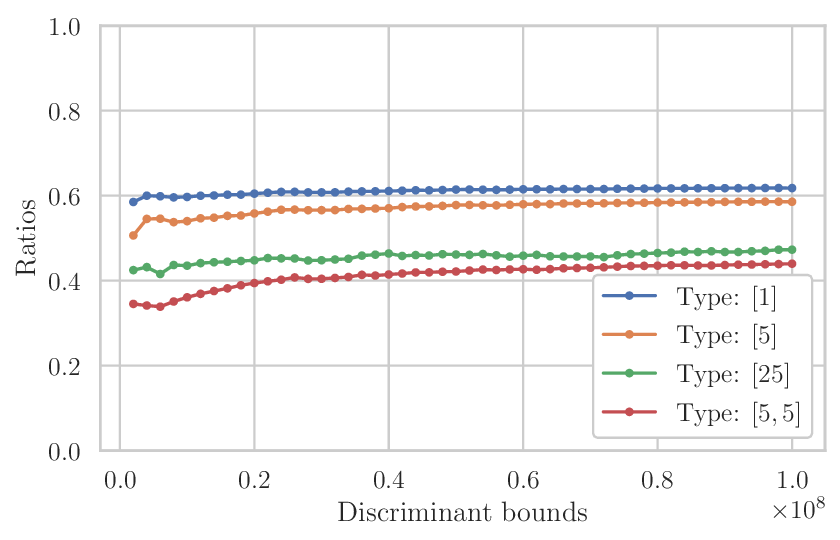}\label{fig:p5}}


\subfloat[$p=7$]{\includegraphics[scale=0.48]{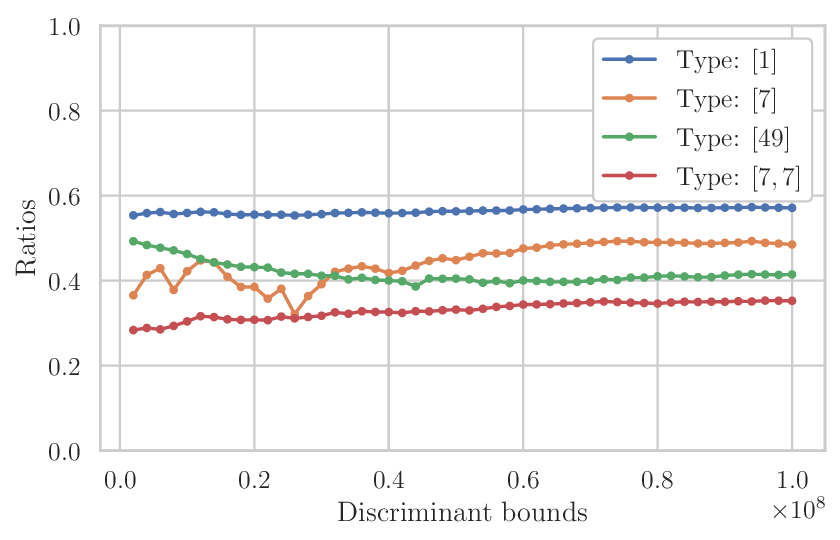}\label{fig:p7}}
\hfill
\subfloat[$p=11$]{\includegraphics[scale=0.48]{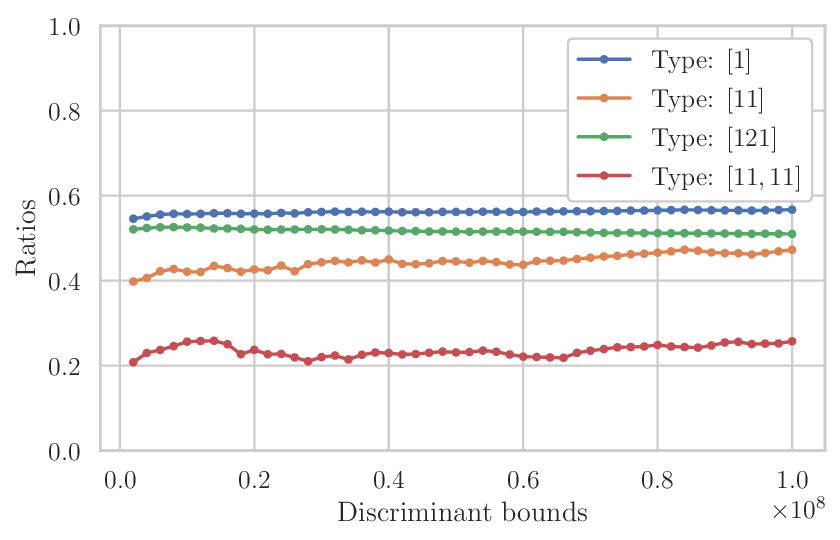}\label{fig:p11}}

\caption{Plots of $\log |E_{-, p}(G, X)|/\log X$ for varying negative discriminant bounds and $p =3, 5, 7, 11$}\label{fig:error_power}
\end{figure}

Our numerical experiments for imaginary quadratic fields with $p=3,5,7,11$ and $|D|<10^8$ reveal that as $X$ grows:
\[
\text{the ratio} ~\frac{\log |E_{-, p}(G, X)|}{\log X} \text{ exhibits clear convergence, with a limit greater than } 1/4.
\]
\begin{remark} While the numerical experiments provide limited evidence for error term predictions in $2^\infty$-class groups, the Cohen-Lenstra heuristic for Selmer groups of quadratic twists of an elliptic curve $E$ offers an illuminating parallel. 
As noted in \cite[Remark 1.3]{smith2023Distributiona}, in this analogous setting, the error term corresponds to twisted curves $E^d$ of Mordell-Weil rank greater than $1$ and is conjectured to be $O(X^{3/4+\epsilon})$ for any $\epsilon>0$ when $|d|\leq X$. 
\end{remark}

\smallskip
We now consider the interpretation of the Cohen-Lenstra heuristic through random matrix models, which will allow us to analyze the error term more precisely.

The connection to random matrices was first established by Friedman and Washington \cite{friedman1989Distribution} for function fields. Following \cite{wood2019Random, venkatesh2011Statistics}, we review how the $p$-class group arises as the cokernel of a random matrix. Let $S$ be a finite set of primes in $\BQ(\sqrt{D})$ generating $\Cl(D)$, with $O_S^\times$ the $S$-unit group and $I_S$ the group of fractional ideals generated by $S$. This gives the exact sequence:
\[
O_S^\times\otimes \BZ_p \rightarrow I_S\otimes \BZ_p \rightarrow \Cl(D)[p^\infty] \rightarrow 0.
\]
Here, $I_S\otimes \BZ_p$ (respectively $O_S^\times\otimes \BZ_p$) is a free $\BZ_p$-module of rank $n:=|S|$ (respectively $n+u_\pm$), allowing us to express $\Cl(D)[p^\infty]$ as $\coker M_D$ for some matrix $M_D\in \Mat_{n\times (n+u_\pm)}(\BZ_p)$.

Crucially, the matrix size $n \geq \dim_{\BF_2}\Cl(D)[2]$ grows asymptotically as $\log\log |D|$, and thus increases with $|D|$ and $X$. 
If we model $M_D$ as random in this limit, we obtain the correspondence:
\begin{equation}\label{equ:cor}
\frac{\#\{0<\pm D<X: \Cl(D)[p^\infty]\simeq G\}}{\#\{0<\pm D<X\}} \leftrightsquigarrow \mu\left(\{M\in \Mat_{n\times (n+u_\pm)}(\BZ_p) : \coker M\simeq G\}\right), \tag{$\ast$}
\end{equation}
where $\mu$ is the normalized Haar measure on $\Mat_{n\times(n+u_{\pm})}(\BZ_p)$. 
This leads to the Cohen-Lenstra distribution through the key result \cite{friedman1989Distribution, wood2019Random}:
\[
\lim_{n\to \infty}\mu\left(\{M\in \Mat_{n\times (n+m)}(\BZ_p) : \coker M\simeq G\}\right) = \frac{\eta_\infty(p)/\eta_{m}(p)}{|G|^{m}|\Aut(G)|}.
\]

\bigskip
In this paper, we establish the following refined version of the random matrix model distribution:
\begin{theorem}\label{thm:main_thm0}
For a prime $p$ and integer $m \geq 0$, consider the normalized Haar measure $\mu$ on $\Mat_{n\times(n+m)}(\BZ_p)$. Then for any finite abelian $p$-group $G$,
\[
\mu\left(\{M \in \Mat_{n\times(n+m)}(\BZ_p) : \coker M \simeq G\}\right) = w_m(G) + \lambda_m(G)p^{-n} + O(p^{-2n}),
\]
where 
\[
w_m(G) = \frac{\eta_\infty(p)/\eta_m(p)}{|G|^m|\Aut(G)|}, \quad 
\lambda_m(G) = \frac{w_m(G)(1 + p^{-m} - p^{\rk_p(G)})}{p-1}.
\]
The implicit constant is at most $\left(\eta_m(p)^2/\eta_\infty(p)^2 - 1\right)^{1/2}$.
\end{theorem}
\begin{remark}
The proof technique actually yields higher-order expansions when needed.
\end{remark}

For further discussion on the application of this random matrix model, we need to fix a method for choosing $S$, that is, choosing the primes that generate $\Cl(D)$. Let $T(D)$ be the smallest value such that the collection of all prime ideals with norm $\leq T(D)$ generates $\Cl(D)$. We take $S$ to be this specific set of prime ideals. By the prime number theorem for number fields, we have the asymptotic relation:
\[
n = |S| \asymp \frac{T(D)}{\log T(D)}.
\]

Consider the correspondence \eqref{equ:cor}. Multiplying both sides by $\#\{0 < \pm D < X\}$ and subtracting the main term $w_{u_\pm}(G)\cdot \frac{3}{\pi^2}X$, we find that $E_{\pm, p}(G, X)$ corresponds to:
\begin{align*}
&\mu\left(\{M \in \Mat_{n\times (n+u_\pm)}(\BZ_p) : \coker M \simeq G\}\right)\cdot\#\{0 < \pm D < X\} - w_{u_\pm}(G)\cdot \frac{3}{\pi^2}X \\
&= w_m(G)O(X^{1/4+\epsilon}) + \lambda_m(G)p^{-n}\frac{3}{\pi^2}X  + \lambda_m(G)O(p^{-n}X^{1/4+\epsilon}) + O(p^{-2n}X).
\end{align*}
Here we use the conjectural error bound $\#\{0 < \pm D < X\} = \frac{3}{\pi^2}X + O(X^{1/4+\epsilon})$ from \cite[Remark 1.1]{nunes2015Squarefree}. 
The dominant contribution to $E_{\pm, p}(G, X)$ comes from comparing two terms: the error term $O(w_{u_\pm}(G)X^{1/4+\epsilon})$ from discriminant counting, and the secondary term $\lambda_{u_\pm}(G)p^{-n}\cdot \frac{3}{\pi^2}X$ from the matrix model. 
The larger of these two terms will dominate.

To analyze the behavior of $X/p^n$, we must consider the growth of $T(D)$. 
Current results in \cite{kim2021Minimal} show $T(D) \ll (\log D)^2$. 
Furthermore, \cite{belabas2007Small} suggests that on average, $T(D)$ may grow more slowly: \textit{"It even looks plausible that the average value of $T(D)$ as $D$ increases is $O((\log D)^{1+\epsilon})$ for any $\epsilon>0$"}. This implies:
\[
\frac{1}{\#\{0 < \pm D < X\}}\sum_{0 < \pm D < X} T(D) \ll (\log X)^{1+\epsilon}.
\]
Thus the average value of $n \asymp \frac{T(D)}{\log T(D)} $ is likely of size $o((\log X)^{1+\epsilon})$ for any $\epsilon>0$, suggesting that $X/p^n$ behaves essentially like a power-saving term.

Combining these analytic and heuristic considerations, we arrive at the following refined conjecture: For every prime $p$ and any finite abelian $p$-group $G$, there exist explicit constants $B_{\pm,p}(G)$ and exponents $\theta_{\pm,p}(G) > \frac{1}{4}$ such that the error term satisfies:
\[
E_{\pm, p}(G, X) \sim B_{\pm,p}(G)X^{\theta_{\pm,p}(G)} \quad \text{as } X \to \infty.
\]
This conjecture naturally combines predictions from random matrix theory with established number-theoretic and numerical evidence.

\medskip 
\subsection{Random matrices over finite fields}
The cokernel distribution of $p$-adic matrices is closely related to the corank distribution of random matrices over finite fields. 
These distributions play a significant role in number theory, particularly in the Cohen-Lenstra conjecture for ideal class groups and Selmer groups \cite{smith2017$2^infty$selmer, koymans2020Effective}, with additional applications in coding theory (cf. \cite{fulman2015Steins}).

In this paper, we investigate several arithmetically significant random matrix models over finite fields, including uniform, symmetric, and skew-symmetric cases. 
These corank distributions share a key feature: they form reversible Markov chains with compactness properties. 
Consequently, analyzing their asymptotic behavior reduces to studying the convergence rates of the associated Markov chains. 

For such chains, exponential convergence occurs precisely when certain drift conditions are satisfied~\cite{gallegos-herrada2024Equivalences}, with the convergence rate determined by the largest absolute value of non-one eigenvalues in the transition matrix \cite{mao2013Spectral}. 
Our approach treats these transition matrices as compact operators on separable Hilbert spaces. 
Using $q$-series techniques, we determine their complete spectra and apply the spectral theorem to obtain detailed asymptotic behavior at all orders.

Let $P$ be the transition matrix of an irreducible, aperiodic Markov chain on a countable set $I$, reversible with respect to $\pi$ (i.e., $\pi(i)P(i,j) = \pi(j)P(j,i)$ for all $i,j\in I$). 
We work in the Hilbert space $\ell^2(\pi)$ of complex-valued sequences $\mu = (\mu(i))_{i\in I}$, equipped with the inner product and norm:
\[
\langle \mu,\nu\rangle_\pi := \sum_{i\in I} \frac{\mu(i)\overline{\nu(i)}}{\pi(i)}, \quad 
\|\mu\|_\pi := \langle\mu,\mu\rangle_\pi^{1/2}.
\]
In this framework, $P$ acts as a bounded, self-adjoint linear operator on $\ell^2(\pi)$ via $P\cdot\mu := \mu P$.

\medskip
Let $q$ be a prime power and $m \geq 0$ a non-negative integer. 
We consider the random variables $\{X_{n,m}\}_{n \geq 1} : \text{Mat}_{n \times (n+m)}(\mathbb{F}_q) \to \mathbb{Z}_{\geq 0}$ defined by $M \mapsto \text{corank}\, M := n - \text{rank}\, M$. Following \cite{gerthiii1986Limit}, we have  
$$
\Prob\left(X_{n,m} = i\right) = (\delta_0 P_m^n)(i),
$$
where $\delta_0 = (1, 0, 0, \dots)$ and the transition matrix $P_m$ is given by  
$$
P_m(i,j) = 
\begin{cases}
q^{-1-2i-m}, & \text{if } j = i+1, \\
1 - (1 - q^{-i})(1 - q^{-m-i}) - q^{-1-2i-m}, & \text{if } j = i, \\
(1 - q^{-i})(1 - q^{-m-i}), & \text{if } j = i-1, \\
0, & \text{otherwise.}
\end{cases}
$$  
The matrix $P_m$ is irreducible, aperiodic, and reversible with respect to the stationary distribution $\pi_m$, where  
$$
\pi_m(i) = \frac{\eta_\infty(q)}{q^{i(i+m)} \eta_i(q) \eta_{i+m}(q)}.
$$  

We analyze the spectral properties of $P_m$ and prove the following theorem.  

\begin{theorem}\label{thm:main_thm_uniform}  
The operator $P_m$ is compact on $\ell^2(\pi_m)$, with eigenvalues $\left\{ q^{-k} : k \geq 0 \right\}$.  
\end{theorem}  

The eigenspaces for each eigenvalue can be explicitly constructed. Moreover, the theorem remains valid for real $q > 1$ and $m > -1$, with the corresponding transition matrix.  

Applying the spectral theorem (Theorem \ref{thm:main_thm_spec}), we obtain the following convergence result, which improves the main theorem in \cite{fulman2015Steins} (see Remark \ref{rmk:comparetoFG15}).  

\begin{corollary}\label{cor:uniform}  
Let $q$ be a prime power and $m \geq 0$ an integer. Then  
$$
\sum_{i=0}^\infty \left| \Prob\left(\corank M = i \mid M \in \text{Mat}_{n \times (n+m)}(\mathbb{F}_q)\right) - \pi_m(i) \right| = \frac{2\pi_m(0)}{(q-1)q^m} \cdot \frac{1}{q^n} + O\left(\frac{1}{q^{2n}}\right),
$$  
where the implicit constant is bounded by $(\pi_m(0)^{-2} - 1)^{1/2}$.  
\end{corollary}  
The asymptotic expansion can also be extended to higher-order terms of $q^{-kn}$ for $k \geq 2$.  

In section \ref{sec:diff-types}, we study other matrix spaces, including symmetric, skew-symmetric, and Hermitian matrices following \cite{fulman2015Steins}. 
The spectra of their associated transition operators are as follows.  

\begin{theorem}  
Let $P_{\sym}$ (resp. $P_{\alt}$, $Q_{\alt}$, $P_{\her}$) denote the transition matrix for symmetric (resp. alternating-1, alternating-2, Hermitian) $n \times n$ matrices, analogous to $P_m$. Then:  
\begin{enumerate}  
    \item $P_{\sym}$ is compact on $\ell^2(\pi_{\sym})$, with eigenvalues $\{\pm q^{-k} : k \geq 0\}\backslash\{-1\}$.  
    \item $P_{\alt}$ (resp. $Q_{\alt}$) is compact on $\ell^2(\pi_{\alt})$ (resp. $\ell^2(\pi'_{\alt})$), with eigenvalues $\{q^{-2k} : k \geq 0\}$.  
    \item $P_{\her}$ is compact on $\ell^2(\pi_{\her})$, with eigenvalues $\{(-q)^{-k} : k \geq 0\}$.  
\end{enumerate}  
\end{theorem}  

These spectral results yield analogous asymptotic expansions for the corank distributions, providing sharp convergence rates.


\subsection*{Acknowledgements}
The authors thank Ye Tian and Jinzhao Pan for helpful comments.
They also thank anonymous referees for valuable suggestions.  The authors thank Peigen Li for helpful discussions and Beijing Institute of Mathematical Sciences and Applications for its support.
The first author was partially supported by the Fundamental Research Funds for the Central Universities (Grant No. XJSJ25010) and the Xiaomi Young Scholar Program.

\section{Spectral theorem on reversible Markov chain}
Let $P$ be a transition matrix defined on a countable set $I$. Assume that $P$ is irreducible and aperiodic, and that $P$ has a unique stationary distribution denoted by $\pi$. According to the basic limit theorem, we have
\[
\|\mu P^n - \pi\|_{tv} \to 0, \quad \text{as } n \to \infty,
\]
for any nonzero initial distribution $\mu$. 
Here, the modified total variation distance (without the factor $1/2$) between two distributions $\mu_1$ and $\mu_2$ is defined as follows:
\[
\|\mu_1 - \mu_2\|_{tv} := \sum_{i \in I} |\mu_1(i) - \mu_2(i)|.
\]
A natural question is how fast $\mu P^n$ converges to $\pi$. Under certain drift conditions (see \cite{gallegos-herrada2024Equivalences} for details), the convergence rate is generally exponential. Can we derive an explicit asymptotic estimate of the convergence rate for specific $P$?
  
\subsection{Reversible Markov chain}
Further assume that $P$ is reversible with respect to $\pi$, i.e., $\pi(j)P(j,i) = \pi(i)P(i,j)$ for any $i, j \in I$.
Since $P$ is irreducible, $\pi(i) > 0$ for all $i$. We define the Hilbert space $\ell^2(\pi)$ of complex-valued sequences as follows:
\[
\ell^2(\pi) = \left\{\mu = (\cdots, \mu(i), \cdots) \in \BC^{I} \,\bigg|\, \sum_{i\in I}\frac{|\mu(i)|^2}{\pi(i)} < \infty\right\}.
\]
The inner product and norm on $\ell^2(\pi)$ are defined as follows:
\[
\langle \mu, \nu \rangle_\pi := \sum_{i \in I} \frac{\mu(i)\overline{\nu(i)}}{\pi(i)}, \quad \|\mu\|_{\pi} := \langle\mu,\mu\rangle^{1/2}.
\]
The Cauchy-Schwarz inequality implies $\|\mu\|_{tv} \leq \|\mu\|_\pi$. Indeed,
\[
\|\mu\|_{tv}^2 = \left(\sum_{i \in I} \frac{|\mu(i)|}{\sqrt{\pi(i)}} \cdot \sqrt{\pi(i)}\right)^2 \leq \|\mu\|_\pi^2.
\]
The operator $P$ naturally acts on $\ell^2(\pi)$ via $P \cdot \mu := \mu P$.

The following spaces are more commonly used in the literature (see \cite{gallegos-herrada2024Equivalences} for details). For $1 \leq p \leq \infty$, define
\[
\ell^p_{\mathrm{old}}(\pi) := \left\{f = (\cdots, f(i), \cdots)^\T \in \BC^{I} \,\bigg|\, \|f\|_{\ell^p} < \infty\right\},
\]
and
\[
P \cdot f := Pf, \quad \|P\|_{\ell^p} := \sup_{\|f\|_{\ell^p} = 1} \|Pf\|_{\ell^p}.
\]
Here, $\|f\|_{\ell^p} := \left(\sum_{i \in I} |f(i)|^p \pi(i)\right)^{1/p}$ for $p \neq \infty$, and $\|f\|_{\ell^\infty} := \sup_{i} |f(i)|$. Note that both $\|P\|_{\ell^1}$ and $\|P\|_{\ell^\infty}$ are no greater than one. By H\"older's inequality, we have $\|P\|_{\ell^2} \leq 1$. 
In particular, $\ell^2_{\mathrm{old}}(\pi)$ is a Hilbert space with the inner product
\[
\langle f, g \rangle_{\ell^2} := \sum_{i \in I} f(i) \overline{g(i)} \pi(i).
\]
Since $P$ is reversible, there exists an isomorphism between the two Hilbert spaces that is compatible with the action of $P$:
\[
\phi: \ell^2(\pi) \to \ell^2_{\mathrm{old}}(\pi), \quad \mu \mapsto (\cdots, \mu(i)/\pi(i), \cdots)^\T.
\]
Then
\[
\|P\|_{\pi} := \sup_{\|\mu\|_{\pi} = 1} \|\mu P\|_{\pi} = \|P\|_{\ell^2} \leq 1.
\]
In other words, $P$ is a linear contraction on $\ell^2(\pi)$.

On the other hand, the adjoint operator $P^\ast$ on $\ell^2(\pi)$ is defined by
\[
P^\ast(i, j) := \frac{P(j,i)\pi(j)}{\pi(i)}.
\]
Thus, $P$ is self-adjoint as an operator. 
By fixing a one-to-one bijection between $I$ and $\mathbb{Z}_{\geq 0}$, we obtain an isomorphism between $\ell^2(\pi)$ and
\[
\ell^2 := \left\{\nu = (\nu(0), \nu(1), \cdots)^\T \in \BC^{\mathbb{Z}_{\geq 0}} \,\bigg|\, \|\nu\| := \left(\sum_{i=0}^{\infty} |\nu(i)|^2\right)^{1/2} < \infty\right\}
\]
by sending $\mu$ to $\mu/\sqrt{\pi}$. Hence, $\ell^2(\pi)$ is separable. In summary, $P$ is a bounded, self-adjoint, linear operator on the separable Hilbert space $\ell^2(\pi)$.

\subsection{Spectral theory}
\begin{definition}[Spectrum of linear operators]
Let $T$ be a linear operator defined on a complex Hilbert space $X$. The spectrum of $T$, denoted by $\sigma(T)$, is defined as follows:
\[
\sigma(T) := \{\lambda \in \mathbb{C} : (\lambda I - T) \text{ is not bijective}\}.
\]
\end{definition}
The spectrum of $T$ is divided into three disjoint subsets:
\begin{enumerate}
  \item[(a)] The point spectrum, or the set of all eigenvalues of $T$, is defined by
\[
\sigma_p(T)  = \{\lambda \in \sigma(T) : \mathrm{Ker}(\lambda I - T) \neq 0\}.
\]
  \item[(b)] The continuous spectrum of $T$ is the set defined by
\[
\sigma_c(T) = \{\lambda \in \sigma(T) : \mathrm{Ker}(\lambda I - T) = 0, \text{ and } \overline{\mathrm{Im}(\lambda I - T)} = X\}.
\]
  \item[(c)] The residual spectrum of $T$ is the set defined by
\[
\sigma_r(T) = \{\lambda \in \sigma(T) : \mathrm{Ker}(\lambda I - T) = 0 \text{ and } \overline{\mathrm{Im}(\lambda I - T)} \subsetneq X\}.
\]
\end{enumerate}

\begin{theorem}[Spectral theorem]
Let $T$ be a bounded self-adjoint linear operator on an infinite-dimensional separable complex Hilbert space $X$. Then
\begin{enumerate}
  \item [$(1)$] $\sigma(T)$ is a closed subset in $B(0, \|T\|)$;
  \item [$(2)$] $\sigma_r(T) = \emptyset$;
  \item [$(3)$] all eigenvalues of $T$ are real;
  \item [$(4)$] eigenvectors associated with distinct eigenvalues are orthogonal.
\end{enumerate}
If $T$ is further assumed to be compact, then
\begin{enumerate}
  \item [$(5)$] all eigenspaces of $T$ are finite-dimensional;
  \item [$(6)$] for any $r > 0$, there are only finitely many eigenvalues of $T$ with absolute value greater than $r$;
  \item [$(7)$] $\sigma(T) = \{0\} \cup \sigma_p(T)$, and at least one of $-\|T\|$ or $\|T\|$ is an eigenvalue of $T$;
  \item [$(8)$] Arrange all eigenvalues by their absolute value: $\|T\| = |\lambda_0| \geq |\lambda_1| \geq |\lambda_2| \geq \cdots$. Then
\[
\ell^2(\pi) = \overline{\bigoplus_{i \geq 0} V_{\lambda_i}},
\]
where $V_{\lambda_i}$ are the eigenspaces associated with $\lambda_i$.
\end{enumerate}
\end{theorem}

\begin{theorem}\label{thm:main_thm_spec}
Assume further that $P$ is a compact operator on $\ell^2(\pi)$. Let $\lambda_0, \lambda_1, \lambda_2, \cdots$ be all eigenvalues of $P$ with non-increasing absolute value. Then for any $\mu \in \ell^2(\pi)$,
\[
\|\mu P^n - \sum_{i=0}^{k} \lambda_i^n \mu_i\|_{tv} = O(|\lambda_{k+1}|^{n}).
\]
Here, $\mu_i$ is the $\lambda_i$-component in the spectral decomposition of $\mu$, and the implicit constant is less than $\|\mu\|_\pi$. In particular,
\[
\|\mu P^n - \mu_0\|_{tv} = \begin{cases}
    \|\mu_1\|_{tv} \cdot |\lambda_1|^{n} + O(|\lambda_2|^{n}), & \text{if } |\lambda_1| > |\lambda_2|, \\
    (\|\mu_1 + (-1)^n \mu_2\|_{tv}) \cdot |\lambda_1|^{n} + O(|\lambda_3|^{n}), & \text{if } |\lambda_1| = |\lambda_2|.
  \end{cases}
\]
Here, $\mu_0 = (\mu \cdot \mathbf{1}) \pi$, $\mathbf{1} = (1, 1,\cdots, 1, \cdots)^\T$, and the implicit constant does not exceed $\left(\|\mu\|_\pi^2 - (\mu \cdot \mathbf{1})^2\right)^{1/2}$.
\end{theorem}
\vspace{\medskipamount}
This theorem generalizes fact 3 in \cite{rosenthal1995Convergence}.
\begin{proof}
Since $P$ is irreducible, it has a unique stationary distribution $\pi$, which is an eigenvector corresponding to the eigenvalue $1$. 
For any $\mu \in V_\lambda$ with $\lambda \neq 1$, we have $\mu \cdot \mathbf{1} = 0$, since $\mu \cdot \mathbf{1} = \mu \cdot P \mathbf{1} = \mu P \cdot \mathbf{1} = \lambda (\mu \cdot \mathbf{1})$. 
We now prove that $\lambda_1 \neq -1$. 
If not, let $\mu$ be a nonzero eigenvector in $V_{-1}$, and decompose $\mu$ as $\mu = \mu_+ - \mu_-$, where $\mu_\pm \geq 0$. 
Without loss of generality, we assume $\sum_{i \in I} \mu_+(i) = 1 = \sum_{i \in I} \mu_-(i)$, since $\sum_{i \in I} \mu(i) = \mu \cdot \mathbf{1} = 0$. 
By the basic limit theorem, $\mu = (-1)^{2n} \mu = \mu P^{2n} = \mu_+ P^{2n} - \mu_- P^{2n}$ converges to $\pi - \pi = 0$ as $n \to \infty$. 
By similar argument, we have $V_1 = \langle \pi \rangle$.

By the spectral theorem, for any $\mu \in \ell^2(\pi)$, we can write $\mu = \sum_{i \geq 0} \mu_i$, where $\mu_i \in V_{\lambda_i}$. Note that $\mu_i$ are orthogonal and $\mu_0 = (\mu \cdot \mathbf{1}) \pi$. Then $\mu P^n = \sum_{i \geq 0} \lambda_i^n \mu_i$, and
\[
\begin{split}
\|\mu P^n - \sum_{i=0}^{k} \lambda_i^n \mu_i\|_\pi^2 &= \sum_{i=k+1}^{\infty} \lambda_i^{2n} \|\mu_i\|_\pi^2 \\
&= \left(\sum_{i=k+1}^\infty \|\mu_i\|_\pi^2 \cdot \left(\frac{\lambda_i}{\lambda_{k+1}}\right)^{2n} \right) \cdot \lambda_{k+1}^{2n} \\
&\leq \left(\sum_{i=k+1}^\infty \|\mu_i\|_\pi^2 \right) \cdot \lambda_{k+1}^{2n}=(\|\mu\|_\pi^2 - \sum_{i=0}^{k} \|\mu_i\|_\pi^2) \cdot \lambda_{k+1}^{2n}.
\end{split}
\]
Thus,
\[
\|\mu P^n - \sum_{i=0}^{k} \lambda_i^n \mu_i\|_{tv} = O(|\lambda_{k+1}|^{n}).
\]
In particular, if $|\lambda_1| > |\lambda_2|$, we have
\[
\|\mu P^n - \mu_0 - \lambda_1^n \mu_1\|_{tv} = O(|\lambda_2|^{n}),
\]
and hence
\[
\|\mu P^n - \mu_0\|_{tv} = \|\mu_1\|_{tv} \cdot |\lambda_1|^{n} + O(|\lambda_2|^{n}).
\]
If $|\lambda_1| = |\lambda_2|$, we have
\[
\|\mu P^n - \mu_0\|_{tv} = (\|\mu_1 + (-1)^n \mu_2\|_{tv}) \cdot |\lambda_1|^{n} + O(|\lambda_3|^{n}), \quad \text{as } n \to \infty.
\]
All implicit constants are bounded above by $(\|\mu\|_\pi^2 - \|\mu_0\|_\pi^2)^{1/2} = (\|\mu\|_\pi^2 - (\mu \cdot \mathbf{1})^2)^{1/2}$.
\end{proof}

\section{Hilbert-Schmidt Markov chains}
Recall that a bounded linear operator $T$ on a separable Hilbert space $X$ is called Hilbert-Schmidt if there exists an orthonormal basis $\{e_n : n \geq 0\}$ such that
\[
\|T\|_{\HS} := \left(\sum_{i=0}^{\infty} \|Te_n\|^2\right)^{\frac{1}{2}} < \infty.
\]
A Hilbert-Schmidt operator is always compact. 
To see this, let $P_N$ be the projection onto the finite-dimensional space spanned by $\{e_1, \cdots, e_N\}$. Then $P_N T$, being a finite-rank operator, is compact and converges to $T$ uniformly.

Note that the Hilbert-Schmidt norm is independent of the choice of orthonormal basis. 
In our situation, the main idea to prove that an operator is Hilbert-Schmidt is to find an orthonormal basis consisting of eigenvectors and then show that
\[
\|T\|_{\HS}^2 = \sum_{i=0}^\infty d_i \lambda_{i}^2 < \infty,
\]
where $d_i$ is the dimension of $V_{\lambda_i}$.

\begin{definition}
Let $P$ be an irreducible, aperiodic, and reversible transition matrix with respect to $\pi$. Then $P$ is said to be Hilbert-Schmidt if it is Hilbert-Schmidt as an operator on $\ell^2(\pi)$.
\end{definition}

For any two real numbers $q > 1$ and $m > -1$, consider the transition matrix $P_m$ on $\mathbb{Z}_{\geq 0}$ defined by
\[
P_m(i, j) = \begin{cases}
  q^{-1 - 2i - m}, & \text{if } j = i + 1, \\
  1 - (1 - q^{-i})(1 - q^{-m - i}) - q^{-1 - 2i - m}, & \text{if } j = i, \\
  (1 - q^{-i})(1 - q^{-m - i}), & \text{if } j = i - 1, \\
  0, & \text{otherwise.}
\end{cases}
\]
This matrix is irreducible and aperiodic because $P_m(i, i) > 0$ for all $i \geq 0$.

Define the distribution $\pi_m$ by
\[
\pi_m(i) = \frac{\theta_m(q)}{q^{i(i + m)} \eta_i(q) \prod_{j=1}^{i} (1 - q^{-m - j})},
\]
where
\[
\eta_k(q) = \prod_{i=1}^{k} (1 - q^{-i}),
\]
and
\[
\theta_m(q)^{-1} := \sum_{i=0}^{\infty} \frac{1}{q^{i(i + m)} \eta_i(q) \prod_{j=1}^{i} (1 - q^{-m - j})} < \infty.
\]
Then $P_m$ is reversible with respect to $\pi_m$.

\bigskip
We now prove that $P_m$ is Hilbert-Schmidt.

\begin{proof}[Proof of Theorem \ref{thm:main_thm_uniform}]
First, we prove that for any $k \geq 1$, the real number $q^{-k}$ is an eigenvalue. 
Define $\pi_m \circ q^i \in \mathbb{C}^{\mathbb{Z}_{\geq 0}}$ by
\[
(\pi_m \circ q^i)(k) := \pi_m(k) \cdot q^{ik}.
\]
We can check that $\pi_m \circ q^i \in \ell^2(\pi_m)$. 
We claim that there exist coefficients $a_0, \dots, a_k$ (depend on $k$) such that $\sum_{i=0}^{k} a_i \cdot (\pi_m \circ q^i)$ is an eigenvector associated with $q^{-k}$.

If $\sum_{i=0}^{k} a_i \cdot (\pi_m \circ q^i)$ is an eigenvector associated with $q^{-k}$, that means for each $l$,
\[
\left(\sum_{i=0}^{k} a_i \cdot (\pi_m \circ q^i) P_m\right)(l) = q^{-k} \left(\sum_{i=0}^{k} a_i \cdot \pi_m(l) \cdot q^{il}\right).
\]
By reversibility and $\pi_m(l)\neq 0$, this is equivalent to
\[
\sum_{j=l-1}^{l+1} \left(\sum_{i=0}^{k} a_i q^{ij}\right) P_m(l, j) = q^{-k} \left(\sum_{i=0}^{k} a_i q^{il}\right).
\]
Since $\sum_{j=l-1}^{l+1} P_m(l, j) = 1$, we have
\[
P_m(l, l-1) \sum_{i=0}^{k} a_i (q^{-i} - 1) q^{il} + P_m(l, l+1) \sum_{i=0}^{k} a_i (q^{i} - 1) q^{il} = (q^{-k} - 1) \left(\sum_{i=0}^{k} a_i q^{il}\right).
\]
Substituting the values of $P_m(l, l-1)$ and $P_m(l, l+1)$, we obtain
\[
\left(1 - (1 + q^{-m}) q^{-l} + q^{-m - 2l}\right) \sum_{i=0}^{k} a_i (q^{-i} - 1) q^{il} + q^{-1 - m - 2l} \sum_{i=0}^{k} a_i (q^{i} - 1) q^{il} = (q^{-k} - 1) \left(\sum_{i=0}^{k} a_i q^{il}\right).
\]
Comparing the coefficients of $q^{il}$ on both sides and formally setting $a_{k+1} = a_{k+2} = 0$, we obtain the recurrence relation
\[
(q^{-i} - q^{-k}) a_i - (1 + q^{-m}) (q^{-1 - i} - 1) a_{i+1} + (q^{-2 - i - m} - q^{-m} + q^{i + 1 - m} - q^{-1 - m}) a_{i+2} = 0, \quad 0 \leq i \leq k.
\]
This recurrence has a unique solution $\{a_0, \dots, a_k\}$ up to a scalar factor.

\bigskip
Next, we show that these eigenvectors generate the entire space $\ell^2(\pi_m)$.

\medskip
One can verify that the equation $v P_m = \lambda v$ has only one solution (up to scale) for each eigenvalue $\lambda$, meaning all eigenspaces $V_\lambda$ are one-dimensional. We claim that the $P_m$-invariant subspace $V := \langle \pi \circ q^i, i \geq 0 \rangle$ is dense in $\ell^2(\pi_m)$, and thus
\[
\ell^2(\pi_m) = \overline{V} = \overline{\bigoplus_{i \geq 0} V_{q^{-i}}}.
\]
To prove this, it suffices to show that $\delta_0 \in \overline{V}$, where $\delta_i \in \ell^2(\pi_m)$ is defined by $\delta_i(k) = 1$ if $k = i$ and $0$ otherwise.
Indeed, if $\delta_0 \in \overline{V}$, then $\delta_1$ also lies in $\overline{V}$ because it is a linear combination of $\delta_0$ and $\delta_0 P_m$. By induction, all $\delta_i$ (which generate $\ell^2(\pi_m)$) belong to $\overline{V}$.

From the $q$-series identity (due to Euler \cite[eq(19)]{gasper2004Basic}), we have
\[
\prod_{i=1}^{\infty} \left(1 - q^{-i} t\right) = \sum_{k=0}^{\infty} \frac{(-1)^k}{\prod_{j=1}^{k} (q^j - 1)} t^k.
\]
Let $b_k = \frac{(-1)^k}{\prod_{j=1}^{k} (q^j - 1)}$ and $c_k = \frac{b_k}{\eta_\infty(q) \pi_m(0)}$. We claim that
\[
\lim_{N \to \infty} \sum_{k=0}^{N} c_k (\pi_0 \circ q^k) = \delta_0 \in \ell^2(\pi_m),
\]
which is equivalent to
\[
\lim_{N \to \infty} \pi_m(0) \cdot \left|\sum_{k=0}^{N} c_k - \pi_m(0)^{-1}\right|^2 + \frac{1}{\eta_\infty(q)^2 \pi_m(0)^2} \cdot \sum_{i=1}^{\infty} \pi_m(i) \cdot \left|\sum_{k=0}^{N} b_k q^{ki}\right|^2 = 0.
\]
By the definition of $c_k$, the first term converges to $0$. Since $\pi_m(i) \ll q^{-i^2 - m i}$ uniformly for all $i$, it remains to show
\[
\lim_{N \to \infty} \sum_{i=1}^{\infty} \frac{1}{q^{i^2 + m i}} \cdot \left|\sum_{k=0}^{N} b_k q^{ki}\right|^2 = 0.
\]
Since $\sum_{k=0}^{\infty} b_k q^{ki} = 0$, we have
\[
\left|\sum_{k=0}^{N} b_k q^{ki}\right| = \left|\sum_{k=N+1}^{\infty} b_k q^{ki}\right|.
\]
Note that
\[
\left|\frac{b_k q^{ki}}{b_{k-1} q^{(k-1)i}}\right| = \frac{q^i}{q^k - 1}.
\]
To use the property of alternating series, we divide the estimation into two parts:
\[
\sum_{i=1}^{N} \frac{1}{q^{i^2 + m i}} \cdot \left|\sum_{k=0}^{N} b_k q^{ki}\right|^2 \quad \text{and} \quad \sum_{i=N+1}^{\infty} \frac{1}{q^{i^2 + m i}} \cdot \left|\sum_{k=0}^{N} b_k q^{ki}\right|^2.
\]

On the one hand, we have
\[
\begin{split}
\sum_{i=1}^{N} \frac{1}{q^{i^2 + m i}} \cdot \left|\sum_{k=N+1}^{\infty} b_k q^{ki}\right|^2 &\leq \sum_{i=1}^{N} \frac{1}{q^{i^2 + m i}} b_{N+1}^2 q^{2(N+1)i} \\
&\leq \frac{1}{\eta_\infty(q)^2} \frac{1}{q^{(N+1)(N+2)}} \sum_{i=1}^{N} \frac{1}{q^{i^2 + m i}} q^{2(N+1)i} \\
&\ll \frac{1}{q^{(N+1)(N+2)}} \sum_{i=1}^{N+1} q^{2(N+1)i - m i - i^2}.
\end{split}
\]
Define $F(N) := \sum_{i=0}^{\infty} q^{2N i - m i - i^2}$ and $f(N) = \frac{F(N)}{q^{N(N+1)}}$. Then
\[
F(N+1) = 1 + q^{2N+1 - m} \sum_{i=1}^{\infty} q^{2N(i-1) - m(i-1) - (i-1)^2} = 1 + q^{2N+1 - m} F(N).
\]
Hence,
\[
f(N+1) = \frac{1}{q^{(N+1)(N+2)}} + \frac{f(N)}{q^{m+1}},
\]
and $f(N) \to 0$ as $N \to \infty$.

On the other hand,
\[
\begin{split}
\sum_{i=N+1}^{\infty} \frac{1}{q^{i^2 + m i}} \cdot \left|\sum_{k=0}^{N} b_k q^{ki}\right|^2 &\leq \sum_{i=N+1}^{\infty} \frac{1}{q^{i^2 + m i}} b_N^2 q^{2N i} \\
&\ll \sum_{i=N+1}^{\infty} q^{-i^2 - m i + 2N i - N(N+1)} \\
&\ll \frac{1}{q^{(m+1)N}} \sum_{i=1}^{\infty} q^{-(i + m/2)^2} \to 0 \text{ as } N \to \infty.
\end{split}
\]
Hence, all normalized eigenvectors form an orthogonal basis of $\ell^2(\pi_m)$, and
\[
\|P_m\|_{\HS}^2 = \sum_{i=0}^{\infty} q^{-2i} = (1 - q^{-2})^{-1}.
\]
Therefore, $P_m$ is Hilbert-Schmidt.
\end{proof}

\begin{remark}\label{rmk:truncation}
The eigenvalues of all transition matrices (both in the theorem above and in the next section) can be estimated numerically using matrix truncation methods (see \cite{kumar2015Truncation}).    
\end{remark}

From the above proof, we can deduce the following lemma, which is crucial when dealing with different types of matrices in next section.

\begin{lemma}\label{lem:key_lemma}
Let $m > -1$ be a real number. If $f(z) = \sum_{i=0}^{\infty} \mu_i z^i\in \BC[[z]]$ satisfies $\sum_{i=0}^{\infty} |\mu_i|^2 q^{i^2 + m i} < \infty$ and $f(q^k) = 0$ for all $k \geq 0$, then $f = 0$.
\end{lemma}

\begin{proof}
Let $\mu = (\cdots, \mu_i, \cdots)$. Note that
\[
\pi_m(k) \asymp \frac{1}{q^{k(k + m)}}.
\]
Thus, $\sum_{i=0}^{\infty} |\mu_i|^2 q^{i^2 + m i} < \infty$ if and only if $\mu \in \ell^2(\pi_m)$. On the other hand, $f(q^k) = 0$ for all $k \geq 0$ is equivalent to
\[
\langle \mu, \pi_m \circ q^i \rangle_{\pi_m} = 0 \quad \text{for all } i \geq 0.
\]
Hence,
\[
\mu \in \langle \pi_m \circ q^i \mid i \geq 0 \rangle^\perp = \ell^2(\pi_m)^\perp = \{0\}.
\]
\end{proof}

\begin{remark}\label{rmk:conterexample} 
\begin{enumerate}
  \item Unlike the conclusion of Carlson's theorem \cite{fuchs1946Generalization} in complex analysis, the main differences are that we cannot control the growth of $f(z)$, and the points $q^k$ are too sparse.
 \item If $m < -1$, the lemma does not hold. Take $\mu_i = b_i q^i$, where $b_i$ is defined in the proof of Theorem \ref{thm:main_thm_uniform}. Then $f(z) = \prod_{k=0}^{\infty} \left(1 - q^{-k} z\right) \neq 0$, but we always have
     \[
     \sum_{i=0}^{\infty} |\mu_i|^2 q^{i^2 + m i} < \infty \quad \text{for any } m < -1.
     \]
\end{enumerate}
\end{remark}

\section{Corank distribution of random matrices}\label{sec:diff-types}
In this section, we investigate specific Markov chains arising from the corank distributions of different types of matrices over finite fields. These problems have been extensively studied in the literature (see~\cite{fulman2015Steins, gerthiii1986Limit}). After proving that these Markov chains are Hilbert-Schmidt, we deduce asymptotic expressions for the corank distributions using the results from earlier sections.

\subsection{Uniform case}
Let $q$ be a prime power and $m$ a non-negative integer. 
The first example we consider is the uniform distribution on the set of all $n \times (n + m)$ matrices over the finite field $\mathbb{F}_q$. 
Define the corank of a matrix $M$ as $\corank M = n - \rank M$. From \cite[Section 1]{gerthiii1986Limit}, we know that
\[
\Prob\left(\corank M = k \mid M \in \Mat_{n \times (n + m)}(\mathbb{F}_q)\right) = (\delta_0 P_m^n)(k),
\]
and the stationary distribution $\pi_m$ is given by
\[
\pi_m(i) = \frac{\eta_\infty(q)}{q^{i(i + m)} \eta_i(q) \eta_{i + m}(q)}.
\]
Note that this $\pi_m$ is the same as the one defined in Section 3.

By Theorems \ref{thm:main_thm_uniform} and \ref{thm:main_thm_spec}, we obtain Corollary \ref{cor:uniform}.

\begin{proof}[Proof of Corollary \ref{cor:uniform}]
By Theorem \ref{thm:main_thm_uniform}, the maximal non-one eigenvalue of $P_m$ is $q^{-1}$, and
\[
\nu := \pi_m - \frac{\pi_m \circ q}{1 + q^{-m}}
\]
is an associated eigenvector (unique up to a scalar since $V_{q^{-1}}$ is one-dimensional). 
The $q^{-1}$-component of $\delta_0$ is given by
\[
(\delta_0)_{q^{-1}} = \frac{\langle \delta_0, \nu \rangle_{\pi_m}}{\langle \nu, \nu \rangle_{\pi_m}} \nu = \frac{(q^m + 1)^{-1}}{\langle \nu, \nu \rangle_{\pi_m}} \nu.
\]

By direct calculation, we have
\[
\begin{split}
\langle \nu, \nu \rangle_{\pi_m} &= \sum_{i=0}^{\infty} \pi_m(i) - \frac{2}{1 + q^{-m}} \sum_{i=0}^{\infty} \pi_m(i) q^i + \frac{1}{(1 + q^{-m})^2} \sum_{i=0}^{\infty} \pi_m(i) q^{2i} \\
&= M(\pi_m, 0) - \frac{2}{1 + q^{-m}} M(\pi_m, 1) + \frac{1}{(1 + q^{-m})^2} M(\pi_m, 2),
\end{split}
\]
where $M(\pi_m, k)$ is the $k$-th moment of $\pi_m$, defined by
\[
M(\pi_m, k) = \sum_{i=0}^{\infty} \pi_m(i) q^{k i}.
\]
From \cite[Example 6.6]{cohen1984Heuristics}, it is known that
\[
M(\pi_m, 0) = 1, \quad M(\pi_m, 1) = 1 + q^{-m}, \quad M(\pi_m, 2) = 1 + (q + 1) q^{-m} + q^{-2m}.
\]
Substituting these values, we obtain
\[
\langle \nu, \nu \rangle_{\pi_m} = \frac{(q - 1) q^m}{(q^m + 1)^2}.
\]

In general, note that $M(\pi_m, k) = (\pi_m \circ q^k) \cdot \mathbf{1}$ and $\sum_{i=0}^{k} a_i (\pi_m \circ q^i) \in V_{q^{-k}}$ for some coefficients $a_i \in \mathbb{R}$. Since $V_\lambda \perp \mathbf{1}$ for $\lambda \neq 1$, we can compute $M(\pi_m, k)$ by induction.

Now, we have
\[
(\delta_0)_{q^{-1}} = \frac{q^m + 1}{(q - 1) q^m} \nu.
\]
Note that $\nu(0) > 0$ and $\nu(i) \leq 0$ for all $i > 0$. 
Since $\nu \cdot\mathbf{1}=\sum_{i=0}^{\infty} \nu(i) = 0$, the total variation norm of $\nu$ is
\[
\|\nu\|_{tv} =2\nu(0)= \frac{2 \pi_m(0)}{q^m + 1}.
\]
Therefore, the total variation norm of $(\delta_0)_{q^{-1}}$ is
\[
\|(\delta_0)_{q^{-1}}\|_{tv} = \frac{2 \pi_m(0)}{(q - 1) q^m}.
\]

By Theorem \ref{thm:main_thm_spec}, we have
\[
\|\delta_0 P_m^n - \pi_m\|_{tv} = \frac{2 \eta_\infty(q) / \eta_m(q)}{(q - 1) q^m} q^{-n} + O(q^{-2n}),
\]
where the implicit constant is bounded above by
\[
\left(\|\delta_0\|_\pi^2 - (\delta_0 \cdot \mathbf{1})^2\right)^{1/2} = \left(\pi_m(0)^{-2} - 1\right)^{1/2}.
\]
\end{proof}

\begin{remark}\label{rmk:comparetoFG15}
In \cite{fulman2015Steins}, Fulman and Goldstein proved that (note that $\|\cdot \|_{tv} = 2 \|\cdot \|_{TV}$)
\[
\frac{1}{4 q^{m + 1}} q^{-n} \leq \|\delta_0 P_m^n - \pi_m\|_{tv} \leq \frac{6}{q^{m + 1}} q^{-n}.
\]
Our estimate improves upon their result, as can be seen from the comparison:
\[
\frac{2 \eta_\infty(q) / \eta_m(q)}{(q - 1) q^m} < \frac{2}{(q - 1) q^m} < \frac{6}{q^{m + 1}},
\]
and
\[
\frac{2 \eta_\infty(q) / \eta_m(q)}{(q - 1) q^m} \geq \frac{2 \eta_\infty(q)}{(q - 1) q^m} \geq \frac{2 \eta_\infty(2)}{(q - 1) q^m} > \frac{1}{4 q^{m + 1}}.
\]
Here, $\eta_\infty(2) \approx 0.29$.
\end{remark}

\bigskip
Now, we transition from the corank distribution to the cokernel distribution.

\begin{proof}[Proof of Theorem \ref{thm:main_thm0}]
For a matrix $M \in \Mat_{n \times (n + m)}(\mathbb{Z}_p)$, recall that the cokernel of $M$ is defined as the quotient $\mathbb{Z}_p^n / \col(M)$, where $\col(M) := M \mathbb{Z}_p^{n + m}$ denotes the submodule of $\mathbb{Z}_p^n$ generated by the columns of $M$. For any finite abelian $p$-group $G$, the probability measure can be expressed as
\[
\mu\left(\left\{ M \in \Mat_{n \times (n + m)}(\mathbb{Z}_p) : \coker M \simeq G \right\}\right) = \sum_{\substack{L \leq \mathbb{Z}_p^n, \\ \mathbb{Z}_p^n / L \simeq G}} \mu(\col^{-1}(L)),
\]
where $L$ runs over submodules of $\mathbb{Z}_p^n$.

Fix an $M_0 \in \col^{-1}(L) \subset \Mat_{n \times (n + m)}(\mathbb{Z}_p)$. Then the preimage of $L$ can be expressed as
\[
\col^{-1}(L) = \left\{ M_0 Q : Q \in \GL_{n + m}(\mathbb{Z}_p) \right\}.
\]

Consider the decomposition $M_0 = P_0 \mathrm{diag}(a_1, a_2, \dots, a_n) Q_0$, where $P_0 \in \GL_n(\mathbb{Z}_p)$, $Q_0 \in \GL_{m + n}(\mathbb{Z}_p)$, and $\mathrm{diag}(a_1, a_2, \dots, a_n) \in \Mat_{n \times (n + m)}(\mathbb{Z}_p)$ is the diagonal matrix with diagonal elements $a_1, a_2, \dots, a_n$. Since $\col(M_0)$ has finite index in $\mathbb{Z}_p^n$, all $a_i$ are nonzero and satisfy $|a_1 \cdots a_n|_p = |G|^{-1}$.

Note that $\mu(P_0 ~\cdot~)$ also defines a Haar measure on $\Mat_{n \times (n + m)}(\mathbb{Z}_p)$ with $\mu(P_0 \Mat_{n \times (n + m)}(\mathbb{Z}_p)) = 1$. The uniqueness of the Haar measure implies $\mu(P_0 ~\cdot~) = \mu$. Thus, we obtain
\[
\mu(\col^{-1}(L)) = \mu\left(\left\{ \mathrm{diag}(a_1, \dots, a_n) Q : Q \in \GL_{n + m}(\mathbb{Z}_p) \right\}\right).
\]

Hence,
\[
\begin{split}
\mu(\col^{-1}(L)) 
&= \mu\left( \left\{ (a_1 \alpha_1, \dots, a_n \alpha_n)^\top : (\alpha_1, \dots, \alpha_{n + m})^\top \in \GL_{n + m}(\mathbb{Z}_p) \right\} \right) \\
&= |a_1|_p^{n + m} \cdots |a_n|_p^{n + m}  \mu\left( \left\{ (\alpha_1, \dots, \alpha_n)^\top : \alpha_i \in \mathbb{Z}_p^{n + m} \setminus \langle p \mathbb{Z}_p^{n + m}, \alpha_1, \dots, \alpha_{i - 1} \rangle \right\} \right) \\
&= |G|^{-(n + m)} \prod_{i = m + 1}^{n + m} (1 - p^{-i}).
\end{split}
\]

Combining this with the submodule counting formula from \cite[Proposition 3.1]{cohen1984Heuristics}:
\[
\sum_{\substack{L \leq \mathbb{Z}_p^n, \\ \mathbb{Z}_p^n / L \simeq G}} 1 = |G|^n |\Aut(G)|^{-1}  \frac{\eta_n(p)}{\eta_{n - r}(p)},
\]
where $r = \rk_p(G) := \dim_{\mathbb{F}_p} G / pG$ denotes the $p$-rank of $G$. We conclude that
\[
\mu\left(\left\{ M : \coker M \simeq G \right\}\right) = |G|^{-m} |\Aut(G)|^{-1} \frac{\eta_{n + m}(p) \eta_n(p)}{\eta_m(p) \eta_{n - r}(p)}.
\]

Recalling the classical results (for example, see \cite{fulman2015Steins}),
\[
\Prob(\corank \overline{M} = r \mid \overline{M} \in \Mat_{n \times (n + m)}(\mathbb{F}_p)) = p^{-r(r + m)} \frac{\eta_{n + m}(p) \eta_n(p)}{\eta_{n - r}(p) \eta_r(p) \eta_{r + m}(p)},
\]
we establish the following relation:
\[
\mu\left(\left\{ M : \coker M \simeq G \right\}\right) = \frac{p^{r(r + m)} \eta_r(p) \eta_{r + m}(p)}{|G|^m |\Aut(G)| \eta_m(p)} \Prob\left(\corank \overline{M} = r \mid \overline{M} \in \Mat_{n \times (n + m)}(\mathbb{F}_p)\right).
\]
The above process originates from \cite[Proposition 1]{friedman1989Distribution} (for $m = 0$) or \cite[Proposition 14.1]{koymans2021Distribution} (for $m = 1$).

Finally, reformulate the cokernel distribution using the Markov chain:
\[
\mu\left(\left\{ M : \coker M \simeq G \right\}\right) = \frac{w_m(G)}{\pi_m(r)} (\delta_0 P_m^n)(r).
\]
Thus, to study the asymptotic behavior of the cokernel distribution as $n \to \infty$, we only need to calculate $(\delta_0 P_m^n)(r)$.

In the proof of Corollary \ref{cor:uniform}, we have the decomposition:
\[
\delta_0 = \pi_m + \frac{p^m + 1}{(p - 1) p^m} \nu + \delta',
\]
where
\[
\nu = \pi_m - \frac{\pi_m \circ p}{1 + p^{-m}} \in V_{p^{-1}} \quad \text{and} \quad \delta' \in \overline{\bigoplus_{i \geq 2} V_{p^{-i}}}.
\]
Hence,
\[
\|\delta' P_m^n\|_{tv} \leq \|\delta' P_m^n\|_{\pi_m} \leq \|\delta'\|_{\pi_m}  p^{-2n} \leq \left(\pi_m(0)^{-2} - 1\right)^{1/2}  p^{-2n},
\]
and so $(\delta' P_m^n)(r) = O(p^{-2n})$.

From this,
\[
\begin{split}
\mu\left(\left\{ M : \coker M \simeq G \right\}\right) &= \frac{w_m(G)}{\pi_m(r)}  \left(\pi_m(r) + \frac{p^m + 1}{(p - 1) p^m} \nu(r)  p^{-n} + O(p^{-2n})\right) \\
&= w_m(G) + \frac{w_m(G) (1 + p^{-m} - p^r)}{p - 1} p^{-n} + O(p^{-2n}).
\end{split}
\]
Here, since $w_m(G)\leq \pi_m(r)$, the implicit constant is bounded above by
\[
\left(\pi_m(0)^{-2} - 1\right)^{1/2} = \left(\eta_m(p)^2/\eta_\infty(p)^2 - 1\right)^{1/2} .
\]
\end{proof}

\medskip
\subsection{Skew centrosymmetric case}
Assume $q$ is odd. Consider the space of skew centrosymmetric matrices:
\[
\Mat^{\scs}_n(\mathbb{F}_q) := \left\{ M \in \Mat_n(\mathbb{F}_q) : M_{ij} = -M_{ji} = M_{n+1-j,n+1-i} \right\}.
\]
Note that the rank of such matrices is always even (see \cite{fulman2015Steins}). 

More precisely, we have the following corank distributions:
\[
\Prob(\corank M = 2k \mid M \in \Mat_{2n}^{\scs}(\mathbb{F}_q)) = \Prob(\corank M = k \mid M \in \Mat_n(\mathbb{F}_q))
\]
and
\[
\Prob(\corank M = 2k+1 \mid M \in \Mat_{2n+1}^{\scs}(\mathbb{F}_q)) = \Prob(\corank M = k \mid M \in \Mat_{n \times (n+1)}(\mathbb{F}_q)).
\]
Thus, these corank distributions can be directly derived from the uniform case results.

\medskip

\subsection{Symmetric case}
Let $q$ be a prime power. Consider the space of symmetric matrices:
\[
\Mat_n^\sym(\mathbb{F}_q) := \{M \in \Mat_n(\mathbb{F}_q) : M^\T = M\}.
\]
From \cite{gerthiii1986Limit}, we have the corank distribution:
\[
\Prob(\corank M = k \mid M \in \Mat_n^\sym(\mathbb{F}_q)) = (\delta_0 P_\sym^n)(k),
\]
where the transition matrix $P_\sym$ is defined by:
\[
P_\sym(i,j) = 
\begin{cases}
q^{-i-1}, & \text{if } j = i+1, \\
q^{-i} - q^{-i-1}, & \text{if } j = i, \\
1 - q^{-i}, & \text{if } j = i-1, \\
0, & \text{otherwise}.
\end{cases}
\]
The Markov chain $P_\sym$ is irreducible, aperiodic, and reversible with stationary distribution:
\[
\pi_\sym(k) = \frac{\alpha(q)}{\prod_{i=1}^k (q^i - 1)}, \quad \alpha(q) = \prod_{\substack{i=1 \\ i \text{ odd}}}^\infty (1 - q^{-i}).
\]

\begin{theorem}\label{thm:sym_spec}
$P_\sym$ is Hilbert-Schmidt on $\ell^2(\pi_\sym)$ with point spectrum:
\[
\sigma_p(P_\sym) = \{\pm q^{-k} : k \geq 0\} \setminus \{-1\}.
\]
\end{theorem}

\begin{proof}
Similar to the approach in Theorem \ref{thm:main_thm_uniform}, the eigenvectors associated with eigenvalues $\pm q^{-k}$ can be expressed as linear combinations of $\pi_m$ and $\{\pi_m\circ (\pm q^i) : i=1,\dots,k\}$, with each eigenspace $V_{\pm q^{-k}}$ being 1-dimensional.

The key step is to prove the spectral decomposition:
\[
\ell^2(\pi_\sym) = \overline{\langle\pi_\sym\rangle \oplus \bigoplus_{k\geq 1} V_{\pm q^{-k}}}.
\]
This reduces to showing that if $\mu\in \ell^2(\pi_\sym)$ satisfies both $\mu\perp \pi_\sym$ and $\mu\perp (\pi_\sym\circ (\pm q^k))$ for all $k\geq 1$, then $\mu=0$.

The orthogonality condition $\mu\perp (\pi_\sym\circ (\pm q^k))$ implies:
\[
\sum_{i=0}^\infty \mu(i)(\pm 1)^i q^{ki} = 0 \quad \text{for all } k\geq 1.
\]
This decouples into two independent conditions:
\[
\sum_{i=0}^\infty \mu(2i)(q^2)^{ki} = 0 \quad \text{and} \quad \sum_{i=0}^\infty \mu(2i+1)(q^2)^{ki} = 0.
\]

Since $\mu\in \ell^2(\pi_\sym)$ is equivalent to $\sum_{i=0}^\infty |\mu(i)|^2 q^{i(i+1)/2} < \infty$, we have:
\[
\sum_{i=0}^\infty |\mu(2i)|^2 (q^2)^{i^2} < \infty \quad \text{and} \quad \sum_{i=0}^\infty |\mu(2i+1)|^2 (q^2)^{i^2} < \infty.
\]

Define the even and odd parts:
\[
\mu_{\text{even}} = (\mu(0),\mu(2),\dots), \quad \mu_{\text{odd}} = (\mu(1),\mu(3),\dots).
\]
These satisfy $\mu_{\text{even}}, \mu_{\text{odd}} \in \ell^2(\pi)$ and are orthogonal to $\overline{\oplus_{k\geq 1}V_{(q^2)^{-k}}}$ in $\ell^2(\pi)$, where $\pi$ is the stationary distribution for the uniform case over $\mathbb{F}_{q^2}$ with $m=0$.

Similar to the proof of the Lemma \ref{lem:key_lemma}, we conclude $\mu_{\text{even}}, \mu_{\text{odd}} \in \langle \pi \rangle$. The condition $\mu\perp \pi_\sym$ implies:
\[
\sum_{i=0}^\infty \mu(i) = 0,
\]
which forces $\mu_{\text{even}} = a\pi = -\mu_{\text{odd}}$ for some $a\in \mathbb{C}$. Thus $\mu\in \langle \hat{\pi} \rangle$, where
\[
\hat{\pi} = (\pi(0),-\pi(0),\pi(1),-\pi(1),\dots).
\]

By self-duality of $P_\sym$, the orthogonality $\mu \perp \langle \pi_\sym, \pi_\sym \circ (\pm q^k): k\geq 1 \rangle$ implies $P_\sym\cdot \mu =\mu P_\sym$ maintains the same orthogonality. Thus $\mu P_\sym \in \langle \hat{\pi} \rangle$. Since $\langle \hat{\pi} \rangle$ is not $P_\sym$-invariant, we must have $\mu = 0$.
\end{proof}

\begin{corollary}\label{cor:sym_converg}
The convergence rate is given by:
\[
\|\delta_0 P_\sym^n - \pi_\sym\|_{tv} = 
\begin{cases}
\frac{2q\alpha(q)}{q^2-1} q^{-n} + O(q^{-2n}), & n \text{ even}, \\
\frac{2q\alpha(q)}{(q^2-1)(q-1)} q^{-n} + O(q^{-2n}), & n \text{ odd},
\end{cases}
\]
with implicit constants are less than $(\alpha(q)^{-2} - 1)^{1/2}$.
\end{corollary}

\begin{proof}
The two dominant eigenvalues (excluding $1$) are $\lambda_+ = q^{-1}$ and $\lambda_- = -q^{-1}$. Following the approach in Corollary \ref{cor:uniform}, we construct the corresponding eigenvectors:

\[\nu_+:=\pi_\sym-\frac{1}{2}(\pi_\sym\circ q)\in V_{q^{-1}}, \quad \nu_-:=\pi_\sym\circ (-q)\in V_{-q^{-1}}.\]

Define the $k$-th moment of $\pi_\sym$ as $M(\pi_\sym, k) := \sum_{i=0}^\infty \pi_\sym(i) q^{ki}$. By induction, we obtain:
\[M(\pi_\sym, 0)=1, ~M(\pi_\sym, 1)=2, ~M(\pi_\sym, 2)=2+2q.\]
These yield the following inner products:
\[\langle\nu_+,\nu_+\rangle_{\pi_\sym}=\frac{q-1}{2},~ \langle\nu_-,\nu_-\rangle_{\pi_\sym}=2+2q,\]
Then the spectral projections of $\delta_0$ are:
\[(\delta_0)_+=\frac{1}{q-1}\nu_+, ~(\delta_0)_-=\frac{1}{2(q+1)}\nu_-.\]


Observe that in $(\delta_0)_+ + (\delta_0)_{-} $,  only the first coordinates is positive, while in $ (\delta_0)_+ - (\delta_0)_{-}$, only the first two coordinates are positive.
Hence, the total variation norm of $(\delta_0)_+ + (-1)^n (\delta_0)_-$ is:
\[
\|(\delta_0)_+ + (-1)^n (\delta_0)_-\|_{tv} = 
\begin{cases}
\frac{2q \alpha(q)}{q^2-1}, & n \text{ even}, \\
\frac{2q \alpha(q)}{(q^2-1)(q-1)}, & n \text{ odd}.
\end{cases}
\]

Applying Theorem \ref{thm:main_thm_spec}, we obtain the final convergence rate:
\[
\|\delta_0 P_\sym^n - \pi_\sym\|_{tv} = 
\begin{cases}
\frac{2q \alpha(q)}{q^2-1} q^{-n} + O(q^{-2n}), & n \text{ even}, \\
\frac{2q \alpha(q)}{(q^2-1)(q-1)} q^{-n} + O(q^{-2n}), & n \text{ odd},
\end{cases}
\]
where the implicit constants are less than $(\alpha(q)^{-2}-1)^{1/2}$.
\end{proof}

\begin{remark}\label{rmk:explainFG15}
This improves Theorem 4.1 in \cite{fulman2015Steins}, clarifying that the parity distinction arises from $P_\sym$ having eigenvalue pairs $\pm q^{-1}$.
\end{remark}

\medskip
\subsection{Alternating case}
Consider alternating (skew-symmetric) matrices:
\[
\Mat_n^\alt(\mathbb{F}_q) := \left\{M \in \Mat_n(\mathbb{F}_q) : M^\T = -M \text{ and } M_{ii} = 0 \text{ for all } i\right\}.
\]
As established in \cite{fulman2015Steins, gerthiii1986Limit}, such matrices always have even rank. The corank distributions are given by:
\[
\Prob(\corank M = 2j+1 \mid M \in \Mat_{2n+1}^\alt) = (\delta_0 P_\alt^n)(j),
\]
\[
\Prob(\corank M = 2j \mid M \in \Mat_{2n}^\alt) = (\delta_0 Q_\alt^n)(j),
\]
with transition matrices:
\[
P_\alt(i,j) = 
\begin{cases}
q^{-4i-3}, & j = i+1, \\
1 - q^{-4i-3} - (1-q^{-2i})(1-q^{-2i-1}), & j = i, \\
(1-q^{-2i})(1-q^{-2i-1}), & j = i-1, \\
0, & \text{otherwise},
\end{cases}
\]
and
\[
Q_\alt(i,j) = 
\begin{cases}
q^{-4i-1}, & j = i+1, \\
1 - q^{-4i-1} - (1-q^{-2i})(1-q^{-2i+1}), & j = i, \\
(1-q^{-2i})(1-q^{-2i+1}), & j = i-1, \\
0, & \text{otherwise}.
\end{cases}
\]
The stationary distributions of $P_\alt$ (resp. $Q_\alt$) is:
\[
\pi_\alt(j) = \frac{\alpha(q)}{q^{2j^2+j}\eta_{2j+1}(q)} \quad \left(\text{resp.}\quad \pi_\alt'(j) = \frac{\alpha(q)}{q^{2j^2-j}\eta_{2j}(q)}\right).
\]

\begin{theorem}
$P_\alt$ (resp. $Q_\alt$) is Hilbert-Schmidt on $\ell^2(\pi_\alt)$ (resp. $\ell^2(\pi'_\alt)$) with point spectrum:
\[
\sigma_p(P_\alt) = \sigma_p(Q_\alt) = \{q^{-2k} : k \geq 0\}.
\]
\end{theorem}

\begin{proof}
Similar to the approach in Theorem \ref{thm:main_thm_uniform}, for $P_\alt$, the eigenvectors corresponding to $q^{-2k}$ are linear combinations of $\{\pi_\alt \circ (q^{2i}) : 0 \leq i \leq k\}$. 
To complete the proof, it suffices to show that if $\mu \in \ell^2(\pi_\alt)$ satisfies:
\[
\sum_{i=0}^\infty |\mu(i)|^2 (q^2)^{i^2 + i/2} < \infty \quad \text{and} \quad \sum_{i=0}^\infty \mu(i)(q^2)^{ki} = 0 \quad \text{for all } k \geq 1,
\]
then $\mu = 0$. This follows directly from Lemma \ref{lem:key_lemma}. The proof for $Q_\alt$ is analogous.
\end{proof}

\begin{corollary}\label{cor:alt_converg}
The convergence rates are:
\[
\|\delta_0 P_\alt^n - \pi_\alt\|_{tv} = \frac{2\alpha(q)}{(q-1)^2(q+1)} q^{-2n} + O(q^{-4n}),
\]
\[
\|\delta_0 Q_\alt^n - \pi_\alt'\|_{tv} = \frac{\alpha(q) q}{(q-1)(q+1)} q^{-2n} + O(q^{-4n}),
\]
with implicit constants less than $(\eta_1(q)^2\alpha(q)^{-2}-1)^{1/2}$ and $(\alpha(q)^{-2}-1)^{1/2}$ respectively.
\end{corollary}

\begin{proof}
Let us first analyze the case for $P_\alt$. We begin by constructing the eigenvector associated with the eigenvalue $q^{-2}$:
\[
\nu := \pi_\alt - \tfrac{q}{q+1}(\pi_\alt \circ q^2) \in V_{q^{-2}}.
\]
Furthermore, we observe that the following combination belongs to the eigenspace $V_{q^{-4}}$:
\[
(1+q^2)(1+q^{-1})(\pi_\alt-\pi_\alt\circ q^2)+\pi_\alt\circ q^4 \in V_{q^{-4}}.
\]

Proceeding by induction, we establish the moments of the stationary distribution:
\[
M(\pi_\alt, 0) = 1, \quad M(\pi_\alt, 2) = 1 + q^{-1}, \quad M(\pi_\alt, 4) = (1+q^2)(1+q^{-1})q^{-1}.
\]
These moment calculations lead to two important results. First, the inner product of $\nu$ with itself:
\[
\langle \nu, \nu \rangle_{\pi_\alt} = \tfrac{q(q-1)}{q+1}.
\]
Second, the total variation norm of $\nu$:
\[
\|\nu\|_{tv} = \tfrac{2\alpha(q)}{(q+1)\eta_1(q)}.
\]

With these preparations, we can now determine the spectral projection of $\delta_0$ onto $V_{q^{-2}}$:
\[
(\delta_0)_{q^{-2}} = q^{-1}(q-1)^{-1} \nu,
\]
which consequently gives:
\[
\|(\delta_0)_{q^{-2}}\|_{tv} = \tfrac{2\alpha(q)}{(q-1)^2(q+1)}.
\]

Turning now to $Q_\alt$, we follow a parallel approach. The corresponding eigenvector is:
\[
\nu' := \pi_\alt' - \tfrac{1}{q+1}(\pi_\alt' \circ q^2) \in V_{q^{-2}}.
\]
Similarly, we identify an element in $V_{q^{-4}}$:
\[
\pi_\alt'-\tfrac{1}{q}(\pi_\alt'\circ q^2)+\tfrac{1}{q(q+1)(q^2+1)}(\pi_\alt'\circ q^4) \in V_{q^{-4}}.
\]

The moment calculations for $Q_\alt$ yield:
\[
M(\pi_\alt', 0) = 1, \quad M(\pi_\alt', 2) = q+1, \quad M(\pi_\alt', 4) = (q+1)(q^2+1).
\]
From these, we derive the key quantities:
\[
\langle \nu', \nu' \rangle_{\pi_\alt'} = \tfrac{q(q-1)}{q+1}, \quad \|\nu'\|_{tv} = \tfrac{\alpha(q) q}{q+1}.
\]

Finally, the spectral projection for $Q_\alt$ satisfies:
\[
(\delta_0)_{q^{-2}} = (q-1)^{-1} \nu', \quad \|(\delta_0)_{q^{-2}}\|_{tv} = \tfrac{\alpha(q) q}{(q-1)(q+1)}.
\]
Then the desired results follows from Theorem \ref{thm:main_thm_spec}.
\end{proof}
\medskip  

 \subsection{Hermitian case}
Let $q$ be a power of an odd prime, and fix $\theta \in \mathbb{F}_{q^2}$ such that $\theta^2 \in \mathbb{F}_q$ but $\theta \notin \mathbb{F}_q$ (see \cite{fulman2015Steins}). Every element $\alpha \in \mathbb{F}_{q^2}$ can be expressed as $\alpha = a + b\theta$ with $a,b \in \mathbb{F}_q$, and we define its conjugate as $\overline{\alpha} = a - b\theta$. 

For a matrix $M = (\alpha_{ij}) \in \Mat_n(\mathbb{F}_{q^2})$, let $M^\ast = (\overline{\alpha_{ji}})$ denote its conjugate transpose. The space of Hermitian matrices is:
\[
\Mat_n^\her(\mathbb{F}_{q^2}) := \{M \in \Mat_n(\mathbb{F}_{q^2}) : M^\ast = M\}.
\]
The corank distribution is given by:
\[
\Prob(\corank M = k \mid M \in \Mat_n^\her(\mathbb{F}_{q^2})) = (\delta_0 P_\her^n)(k),
\]
where the transition matrix $P_\her$ has entries:
\[
P_\her(i,j) = 
\begin{cases}
q^{-2i-1}, & j = i+1, \\
q^{-2i}(1-q^{-1}), & j = i, \\
1-q^{-2i}, & j = i-1, \\
0, & \text{otherwise},
\end{cases}
\]
with stationary distribution:
\[
\pi_\her(j) = \frac{\beta(q)}{q^{j^2} \eta_j(q^2)}, \quad \beta(q) = \prod_{i=1,~ \text{odd}}^\infty (1 + q^{-i})^{-1}.
\]

\begin{theorem}
The operator $P_\her$ is Hilbert-Schmidt on $\ell^2(\pi_\her)$ with point spectrum:
\[
\sigma_p(P_\her) = \{(-q)^{-k} : k \geq 0\}.
\]
\end{theorem}

\begin{proof}
Following the approach in Theorem \ref{thm:main_thm_uniform}, we need to verify:
\[
\lim_{N\to\infty} \sum_{i=1}^\infty \frac{1}{q^{i^2}} \left|\sum_{k=0}^N b_k' (-q)^{ki}\right|^2 = 0,
\]
where $b_k' = \frac{(-1)^k}{\prod_{j=1}^k ((-q)^j - 1)}$. 
The proof decomposes into two cases:
\[\lim_{N\ra \infty}\sum_{i=1, odd}^{\infty}\frac{1}{q^{i^2}}\cdot \left|\sum_{k=0}^{N}b_k'(-q)^{ki}\right|^2=0=\lim_{N\ra \infty}\sum_{i=1, even}^{\infty}\frac{1}{q^{i^2}}\cdot \left|\sum_{k=0}^{N}b_k'(-q)^{ki}\right|^2.\]

For odd $i$,   
\[\sign(b_k'(-q)^{ki})\text{ is }
\begin{cases}
    >0, & \mbox{if } k\equiv 0,3\bmod 4 \\
    <0, & \mbox{if } k\equiv 1,2\bmod 4.
   \end{cases} \]
Using the ratio test:
\[
\left|\frac{b_{k}'(-q)^{ki}}{b_{k-1}'(-q)^{(k-1)i}}\right| = \frac{q^i}{q^k - (-1)^k},
\]
we establish convergence via alternating series estimates. Indeed,
we have
\[\sum_{i=1, odd}^{N} \frac{1}{q^{i^2}}\cdot \left|\sum_{k=N+1}^{\infty}b_k'(-q)^{ki}\right|^2   \ll \frac{1}{q^{(N+1)(N+2)}} \sum_{i=1, odd}^{N}\frac{1}{q^{i^2}} q^{2(N+1)i}\ra 0 \text{ as } N\ra \infty,\]
and
\[\sum_{i=N+1,odd}^{\infty}\frac{1}{q^{i^2}}\cdot \left|\sum_{k=0}^{N}b_k'(-q)^{ki}\right|^2  \ll \frac{1}{q^N}\ra 0 \text{ as } N\ra \infty.\]

The even $i$ case follows similarly.

\end{proof}

\begin{corollary}\label{cor:her_converg}
The convergence rate is:
\[
\|\delta_0 P_\her^n - \pi_\her\|_{tv} = \frac{2\beta(q)}{(q+1)\alpha(q^2)} q^{-n} + O(q^{-2n}),
\]
with implicit constant less than $ ( \beta(q)^{-2}-1)^{1/2}$.
\end{corollary}

\begin{proof}
The dominant eigenvector is:
\[
\nu := \pi_\her \circ (-q) \in V_{-q^{-1}}.
\]
Combing with $\pi_\her-\frac{1}{q+1}(\pi_\her\circ q^2)\in V_{q^{-2}}$, we get the inner product $\langle \nu, \nu \rangle_{\pi_\her} = q+1$.
Since $\nu \perp \pi_\her$,  the total variation norm is:
\[\|\nu\|_{tv} = 2\sum_{j=0}^{\infty}|\nu(2j+1)|=2\sum_{j=0}^{\infty} \frac{\beta(q)}{q^{(2j+1)^2} \eta_{2j+1}(q^2)}q^{2j+1}=\frac{2\beta(q)}{\alpha(q^2)}\sum_{j=0}^{\infty}\frac{\alpha(q^2)}{(q^2)^{2j^2+j}\eta_{2j+1}(q^2)} \]
Observe that $\left(\cdots,\frac{\alpha(q^2)}{(q^2)^{2j^2+j}\eta_{2j+1}(q^2)},\cdots\right)$ is the stationary distribution of $P_{\alt}$ over $\BF_{q^2}$.
Hence,
\[\|\nu\|_{tv}=\frac{2\beta(q)}{\alpha(q^2)}.\]

The spectral projection yields:
\[
(\delta_0)_{-q^{-1}} = \frac{1}{q+1} \nu, \quad \|(\delta_0)_{-q^{-1}}\|_{tv} = \frac{2\beta(q)}{(q+1)\alpha(q^2)}.
\]
Then the desired results follows from Theorem \ref{thm:main_thm_spec}.
\end{proof}


\begin{thebibliography}{99}

\bibitem{belabas2007Small}
Belabas K, Diaz y Diaz F, Friedman E. Small generators of the ideal class group. Math Comp, 2008, 77: 1185-1197

\bibitem{bhargava2024Improved}
Bhargava M, Taniguchi T, Thorne F. Improved error estimates for the Davenport-Heilbronn theorems. Math Ann, 2024, 389: 3471-3512.

\bibitem{cohen1984Heuristics}
Cohen H, and Lenstra H W. Heuristics on class groups of number fields. In: Proceedings of the Journees Arithmetiques held at Noordwijkerhout, 1983. Lecture Notes in Mathematics, vol. 1068. Berlin: Springer, 1984, 33-62

\bibitem{martinet1990Étude}
Cohen H, Martinet J. \'Etude heuristique des groupes de classes des corps de nombres. J Reine Angew Math, 1990, 404: 39-76

\bibitem{davenport1971density}
Davenport H, Heilbronn H A. On the density of discriminants of cubic fields. II. Proc R Soc Lond A, 1971, 322: 405-420

\bibitem{fouvry20074rank}
Fouvry \'E, and Kl\"uners J. On the 4-rank of class groups of quadratic number fields. Invent Math, 2007, 167: 455-513

\bibitem{friedman1989Distribution}
Friedman E, and Washington L C. On the distribution of divisor class groups of curves over a finite field. In: Proceedings of the International Number Theory Conference held at Universit\'e Laval, 1987. Th\'eorie des nombres / Number Theory. Berlin: De Gruyter, 1989, 227-239 

\bibitem{fuchs1946Generalization}
Fuchs W H J. A generalization of Carlson's theorem. J Lond Math Soc, 1946, 1: 106-110

\bibitem{fulman2015Steins}
Fulman J, Goldstein L. Stein's method and the rank distribution of random matrices over finite fields. Ann Probab, 2015, 43: 1274-1314

\bibitem{gallegos-herrada2024Equivalences}
Gallegos-Herrada M A, Ledvinka D, and Rosenthal J S. Equivalences of Geometric Ergodicity of Markov Chains. J Theor Probab, 2023, 37: 1230-1256

\bibitem{gasper2004Basic}
Gasper G, Rahman M. Basic hypergeometric series. Cambridge: Cambridge University Press, 1990

\bibitem{gerth19844class}
Gerth III F. The 4-class ranks of quadratic fields. Invent Math, 1984, 77: 489-515

\bibitem{gerthiii1986Limit}
Gerth III F. Limit probabilities for coranks of matrices over $GF(q)$. Linear  Multilinear A, 1986, 19: 79-93

\bibitem{hardy1917Normal}
Hardy G H, Ramanujan S. The normal number of prime factors of a number $n$. Quarterly J Math, 1917, 48: 76-92

\bibitem{kim2021Minimal}
Kim H H. Minimal generators of the ideal class group. J Number Theory, 2021, 222: 157-167

\bibitem{koymans2021Distribution}
Koymans P, Pagano C. On the distribution of $Cl(K)[\ell^\infty]$ for degree $\ell$ cyclic fields. J Eur Math Soc, 2022, 24: 1189-1283

\bibitem{koymans2020Effective}
Koymans P, Pagano C. Effective convergence of coranks of random R\'edei matrices. Acta Arith, 2024, 212: 337-358

\bibitem{koymans2024Bounds}
Koymans P, Thorner J. Bounds for moments of $\ell $-torsion in class groups. Math Ann, 2024, 390: 3221-3237

\bibitem{kumar2015Truncation}
Kumar K. Truncation method for random bounded self-adjoint operators. Banach J Math Anal, 2015, 9: 98-113

\bibitem{lewis2019Numerical}
Lewis C, Williams C. Numerical secondary terms in a Cohen-Lenstra conjecture on real quadratic fields. Involve, 2019, 12: 221-233

\bibitem{mao2013Spectral}
Mao Y H, Song Y H. Spectral gap and convergence rate for discrete-time Markov chains. Acta Math Sin-English Ser, 2013, 29: 1949-1962

\bibitem{nunes2015Squarefree}
Nunes R M. Squarefree numbers in arithmetic progressions. J Number Theory, 2015, 153: 1-36

\bibitem{rosenthal1995Convergence}
Rosenthal J S. Convergence Rates for Markov Chains. Siam Rev, 1995, 37: 387-405

\bibitem{smith2017$2^infty$selmer}
Smith A. $2^\infty$-Selmer groups, $2^\infty$-class groups, and Goldfeld's conjecture. arXiv:1702.02325, 2017

\bibitem{smith2023Distribution}
Smith A. The distribution of $\ell^\infty$-selmer groups in degree $\ell$ twist families I: fixed point selmer groups. arXiv:2207.05674, 2022

\bibitem{smith2023Distributiona}
Smith A. The distribution of $\ell^\infty$-selmer groups in degree $\ell$ twist families II: fixed point selmer groups. arXiv:2207.05143, 2022

\bibitem{venkatesh2011Statistics}
Venkatesh A, Ellenberg J S. Statistics of Number Fields and Function Fields. In: Proceedings of the International Congress of Mathematicians 2010. New Delhi: Hindustan Book Agency, 2011, 383-402

\bibitem{wang2021Moments}
Wang W, Wood M M. Moments and interpretations of the Cohen-Lenstra-Martinet heuristics. Comment Math Helv, 2021, 96: 339-387  

\bibitem{wood2019Random}
Wood M M. Random integral matrices and the Cohen-Lenstra heuristics. Am J Math, 2019, 141: 383-398 







\end{thebibliography}
\end{document}